\documentclass[11pt]{amsart}
\usepackage{hyperref}

\usepackage{enumerate}

\usepackage[utf8]{inputenc}
\usepackage[T1]{fontenc}

\usepackage[mathscr]{eucal}

\usepackage[a4paper, hmargin={3cm,3cm}, centering]{geometry}

\everymath{\displaystyle}

\newtheorem{lemma}{Lemma}[section]
\newtheorem{theorem}[lemma]{Theorem}
\newtheorem*{theorem*}{Theorem}
\newtheorem{corollary}[lemma]{Corollary}

\newtheorem{proposition}[lemma]{Proposition}
\newtheorem*{proposition*}{Proposition}

\theoremstyle{remark}

\theoremstyle{definition}
\newtheorem*{definition*}{Definition}
\newtheorem*{conjecture*}{Conjecture}
\newtheorem*{remark*}{Remark}
\newtheorem*{remarks*}{Remarks}
\newtheorem*{claim*}{Claim}

\newcommand{\C}{{\mathbb C}}
\newcommand{\E}{{\mathbb E}}

\newcommand{\N}{{\mathbb N}}
\renewcommand{\P}{{\mathbb P}}
\newcommand{\Q}{{\mathbb Q}}
\newcommand{\R}{{\mathbb R}}
\newcommand{\T}{{\mathbb T}}
\newcommand{\Z}{{\mathbb Z}}

\newcommand{\norm}[1]{\left\Vert #1\right\Vert}

\begin{document}

\title[Closest integer polynomial multiple recurrence along shifted primes]{Closest integer polynomial multiple recurrence along shifted primes}

\author{Andreas Koutsogiannis}
\address[Andreas Koutsogiannis]{The Ohio State University, Department of mathematics, Columbus, Ohio, USA} \email{koutsogiannis.1@osu.edu}

\begin{abstract}
Following an approach presented by N.~Frantzikinakis, B.~Host and B.~Kra, we show that the parameters in the multidimensional Szemer{\'e}di theorem for closest integer polynomials have non-empty intersection with the set of shifted primes $\P-1$ (or similarly of $\P+1$). Using the Furstenberg Correspondence Principle, we show this result by recasting it as a polynomial multiple recurrence result in measure ergodic theory. Furthermore, we obtain integer part polynomial convergence results by the same method, which is a transference principle that enables one to deduce results for $\Z$-actions from results for flows. We also give some applications of our approach on Gowers uniform sets.
\end{abstract}



\maketitle

\section{Introduction and main results}

 For a subset $E\subset \Z^\ell$ we denote its {\em upper Banach density} by $d^\ast(E),$ which is defined to be the number $$d^\ast(E)=\limsup_{|I|\to\infty}\frac{|E\cap I|}{|I|},$$ where the $\limsup$ is taken over all the parallelepipeds $I\subset \Z^\ell$ whose side lengths tend to infinity.

\medskip

 Bergelson and Leibman, in \cite{BL2}, proved the following polynomial multidimensional Szemer{\'e}di theorem:
 
\begin{theorem*}\label{T:b}
Let $\ell,$ $m\in \N,$ $\vec{q}_1,\ldots,\vec{q}_m:\Z\to\Z^\ell$ be polynomials with $\vec{q}_i(0)=\vec{0}$ for $1\leq i\leq m$ and let $E\subseteq \Z^\ell$ with $d^\ast(E)>0.$ Then there exists $n\in\N$ such that  
$$
d^\ast \left( E\cap(E+\vec{q}_1(n))\cap\ldots\cap (E+\vec{q}_m(n))\right)>0.$$
\end{theorem*} 
 
 Let $\P$ denote the set of prime numbers. Frantzikinakis, Host and Kra showed,  in \cite{FHK}, that the parameters of this result can be restricted to the shifted primes $\P-1$ (and similarly to $\P+1$), generalizing results due to S{\'a}rk{\"o}zy (\cite{Sark}), who showed that any $E-E,$ with $E\subset \Z,$ $d^\ast(E)>0,$ contains a shifted prime $p-1$ (and similarly a shifted prime $p+1$) for some $p\in\P;$ due to Wooley and Ziegler (\cite{WZ}) who proved the $\ell=1$ case and finally, of Bergelson, Leibman and Ziegler (\cite{BLZ}) who got the result for linear polynomials.
 
\medskip 

If $[[\cdot]]$ denotes the closest integer, i.e, $[[x]]=[x+1/2],$ where $[\cdot]$ is the integer part function, for $\vec{q}=(q_1,\ldots,q_m)$ we set $[[\vec{q}]]=([[q_1]],\ldots,[[q_m]]).$

\medskip
 
Following the approach presented in \cite{FHK} and methods from \cite{K}, we will prove the analogous result of Frantzikinakis, Host and Kra for the respective closest integer polynomials. Namely, we will prove the following:

\begin{theorem}\label{T:p}
Let $\ell,$ $m\in \N,$ $\vec{q}_1,\ldots,\vec{q}_m:\Z\to\R^\ell$ be polynomials with $\vec{q}_i(0)=\vec{0}$ for $1\leq i\leq m$ and let $E\subseteq \Z^\ell$ with $d^\ast(E)>0.$ Then the set of integers $n$ such that  
$$
d^\ast \left( E\cap(E+[[\vec{q}_1(n)]])\cap\ldots\cap (E+[[\vec{q}_m(n)]])\right)>0$$ has non-empty intersection with $\P-1$ and $\P+1.$
\end{theorem}

In order to obtain this result, we will make use of the Furstenberg Correspondence Principle (see below). Via this principle, we will get Theorem~\ref{T:p}, by proving a multiple recurrence result in ergodic theory (see Theorem~\ref{T:e} below).

\begin{definition*}
For $\ell \in \N,$ we call the setting $(X,\mathcal{X},\mu,T_1,\ldots,T_\ell)$  a \emph{system}, where $T_1,\ldots, T_\ell\colon$ $ X\to X$ are invertible commuting measure preserving transformations on the probability space $(X,\mathcal{X},\mu).$ 
\end{definition*}

\begin{theorem*}[Furstenberg Correspondence Principle, \cite{Fu}]\label{T:f}
Let $\ell\in \N$ and $E\subseteq \Z^\ell.$ There exist a system $(X,\mathcal{X},\mu, T_1,\ldots,T_\ell)$ and a set $A\in \mathcal{X}$ with $\mu(A)=d^\ast(E),$ such that 
$$
d^\ast \left( E\cap(E+\vec{n}_1)\cap\ldots\cap (E+\vec{n}_m)\right)\geq \mu\left(A\cap(\prod_{i=1}^\ell T_i^{-n_{i,1}})A\cap\ldots\cap (\prod_{i=1}^\ell T_i^{-n_{i,m}})A\right)$$ for all $m\in \N$ and $\vec{n}_j=(n_{1,j},\ldots,n_{\ell,j})\in\Z^\ell$ for $1\leq j\leq m.$
\end{theorem*}

By the Furstenberg Correspondence Principle, in order to prove Theorem~\ref{T:p}, it suffices to prove the following ergodic reformulation of it:

\begin{theorem}\label{T:e}
Let   $\ell,m\in \N$ and  $q_{i,j}\in \R[t]$ be polynomials with $q_{i,j}(0)=0$ for $1\leq i\leq \ell,$ $1\leq j\leq m.$ 
Then, for every system
  $(X,\mathcal{X},\mu,T_1,\ldots, T_\ell)$ and any $A\in \mathcal{X}$ with $\mu(A)>0,$ the set of integers $n$ such that
  $$
  \mu\left(A\cap(\prod_{i=1}^\ell T_i^{-[[q_{i,1}(n)]]})A\cap\ldots\cap (\prod_{i=1}^\ell T_i^{-[[q_{i,m}(n)]]})A\right)>0$$
  has non-empty intersection with $\P-1$ and $\P+1.$
\end{theorem}

\begin{remark*}
 As in \cite{FHK}, the arguments will show that the aforementioned intersection has positive measure for a set of positive relative density in the shifted primes (and so, the same holds for the conclusion of Theorem~\ref{T:p} as well). 
\end{remark*}

Our method will show though that we can have integer part polynomial multiple convergence along shifted primes. Namely, we prove: 

\begin{theorem}\label{T:t3}
Let $\ell, m\in \N,$ $(X,\mathcal{X},\mu,T_1,\ldots, T_\ell)$ be a system, $q_{i,j}\in\R[t]$ polynomials, $1\leq i\leq \ell,$ $1\leq j\leq m$ and  $f_1,\ldots,f_m\in L^{\infty}(\mu)$ be functions with $\norm{f_i}_{\infty}\leq 1.$ 
 Then the averages $$
\frac{1}{\pi(N)}\sum_{p\in\P\cap [1,N]}(\prod_{i=1}^\ell T_i^{[q_{i,1}(p)]})f_1\cdot\ldots\cdot (\prod_{i=1}^\ell T_i^{[q_{i,m}(p)]})f_m$$ converge in $L^2(\mu)$ as $N\to\infty,$ where $\pi(N)=|\P\cap [1,N]|$ denotes the number of primes up to $N$ and $[1,N]$ denotes the set $\{1,\ldots,N\}.$
\end{theorem} 

So, we generalize Theorem~1.3 from \cite{FHK} for integer part polynomial iterates and Theorem~1.1 from \cite{WS} for general polynomials and commuting transformations. Also, since $[[x]]=[x+1/2],$ Theorem~\ref{T:t3}, holds for least integer polynomial iterates as well.

\begin{remark*}
 We cannot in general obtain polynomial multiple recurrence results with iterates given by integer part polynomials (i.e, the conclusion of Theorem~\ref{T:e}, and consequently the conclusion of Theorem~\ref{T:p}). Since $[-x]=-[x]-1$ for $x\notin\Z,$ $x<0,$ for $a\notin \Q,$ we can find  examples with $\mu(T^{[an]}A\cap T^{[-an]}A)=0$ for every $n\in \N.$ 
\end{remark*}

Even if it's not stated, without loss of generality, for the bounded functions $f_i,$ which appear in our expressions, we will always assume that $\norm{f_i}_\infty\leq 1$ for all $i.$

\medskip

Since we recover the convergence results of \cite{FHK} (in the special case where our polynomials take values in $\Z$), we generalize convergence results of Wierdl (\cite{Mate2}) and of Wooley and Ziegler (\cite{WZ}).

\medskip

We remark that Theorems~\ref{T:p}, \ref{T:e} and \ref{T:t3}, for $\ell=1$ and for iterates of the form $[j\alpha n]$ for $1\leq j\leq m$ and $\alpha\in \R\setminus \Q,$ or of the form $j[\alpha n]$ for $1\leq j\leq m$ and $\alpha\in \R\setminus \Q,$ were treated in \cite{WS} by W.~Sun. The proofs of these results will be simplified and extended by our method in the more general setting of commuting transformations. More specifically, the respective result of Theorem~\ref{T:t3} (\cite[Theorem 1.1]{WS}), follows immediately from the linear case of Theorem~\ref{T:t3}. The respective result of Theorem~\ref{T:e}, and so, of Theorem~\ref{T:p} (\cite[Theorem 1.2]{WS} and \cite[Corollary 1.3]{WS} respectively), follows from the following multiple recurrence result on integer part monomials with specific coefficients:

\begin{theorem}\label{T:wg}
Let   $\ell,m\in \N,$  $a\in \R,$ $d_{i,j}\in\N$ positive integers and $k_{i,j}\in\N\cup\{0\}$ be non-negative integers for $1\leq i\leq \ell,$ $1\leq j\leq m.$ 
Then, for every system
  $(X,\mathcal{X},\mu,T_1,\ldots, T_\ell)$ and any $A\in \mathcal{X}$ with $\mu(A)>0,$ the set of integers $n$ such that
  $$
  \mu\left(A\cap(\prod_{i=1}^\ell T_i^{-[ak_{i,1}n^{d_{i,1}}]})A\cap\ldots\cap (\prod_{i=1}^\ell T_i^{-[ak_{i,m}n^{d_{i,m}}]})A\right)>0$$
  has non-empty intersection with $\P-1$ and $\P+1.$
\end{theorem}

Finally, we give an application of our approach on some recent work of Franzikinakis and Host (\cite{FranH}) on Gowers uniform sets. More specifically, we prove (see Section 5 for definitions and details):

\begin{theorem}\label{P:2}
Any shift of a Gowers uniform set $S\subseteq \N$ is a set of closest integer polynomial multiple recurrence and integer part polynomial multiple mean convergence.
\end{theorem}

At this point we will borrow some examples from \cite{FranH} of Gowers uniform sets, and so, sets for which we have the conclusion of Theorem~\ref{P:2}.

\begin{remarks*}[\cite{FranH}]
If $\omega(n)$ is the number of distinct prime factors of an integer $n$ and $\Omega(n)$ the number of prime factors of $n$ counted with multiplicity, then:

\medskip

\noindent i)  Any shift of the sets
$$ S_{\omega,A,b} := \{n \in \N:\; \omega(n)\equiv a (\textit{mod}\; b) \;\textit{for some a} \in A\},$$
and similarly $S_{\Omega,A,b},$ for every $b\in\N$ and $A\subseteq \{0, \ldots, b-1\}$ non-empty, is Gowers uniform.

\medskip

\noindent ii) Any shift of the sets 
$$S_{\omega,A,\alpha} := \{n \in\N:\; \norm{\omega(n)\alpha} \in A\},$$ and similarly $S_{\Omega,A,\alpha},$ for every irrational $\alpha$ and Riemann-measurable set $A \subseteq [0, 1/2]$ of
positive measure, where $\norm{x}:=d(x,\Z)$ denotes the distance to the closest integer, is Gowers uniform.

\medskip

\noindent iii) The same results to ii) hold if $S_{\omega,A,\alpha}$ and $S_{\Omega,A,\alpha}$ are defined using fractional parts.
\end{remarks*} 






\noindent {\bf Notation.} We always denote by $[\cdot],$ $\{\cdot\},$ $[[\cdot]]$ and $\norm{\cdot}$ the integer part, the fractional part, the closest integer and the distance to the closest integer function respectively.
  We denote by $\N$ the set of positive integers and by $\Z_N=\Z/N\Z$ the integers modulo $N.$ When we need, we identify the set $\Z_N$ with the set $[1,N].$ For a finite set $F$ and $a:F\to \C,$ we write $\E_{n\in F}a(n)=\frac{1}{|F|}\sum_{n\in F}a(n).$ For a measurable function $f$ on a measure space $X$ with a transformation $T:X\to X,$ we denote with $Tf$ the composition $f\circ T.$ Given transformations $T_i:X\to X,$ $1\leq i\leq \ell,$ with $\prod_{i=1}^\ell T_i$ we denote the composition $T_1\circ \cdots\circ T_\ell.$
   With $(a,b)$ we denote the greatest common divisor of $a,b.$ A quantity that goes to $0$ as $N\to\infty$ is denoted as $o_N(1),$ while a quantity that goes to $0$ as $N\to \infty$ and then $w\to\infty$ as $o_{N\to\infty;w}(1).$  
  

\subsection*{ Acknowledgements.} I would like to express my gratitude to N.~Frantzikinakis for his in-depth suggestions and the fruitful discussions we had during the writing procedure of this article. These discussions corrected mistakes from the initial version of this paper and led me to additional (reflected mainly in Section 5) results.

\section{Definitions and tools}

In this section we give the definitions and the main ideas in order to prove Theorem~\ref{T:e}. For reader's convenience, we repeat most of the part of Section 2 from \cite{FHK}.

\medskip

We start by recalling the definition of the {\em von Mangoldt function}, $\Lambda:\N\to\R,$ where $ \Lambda(n)=\left\{ \begin{array}{ll} \log(n) \quad ,\text{ if } n=p^k \text{ for some } p\in \P \text{ and some }k\in \N\\ 0 \quad \; \quad \quad ,\text{ elsewhere }\end{array} \right.$. 

\medskip

As in \cite{FHK}, it is more natural for us to deal, in stead of $\Lambda,$ with the function $\Lambda':\N\to\R,$ where $\Lambda'(n)={\bf 1}_{\P}(n)\cdot \Lambda(n)={\bf 1}_{\P}(n)\cdot \log(n).$

\medskip

The function $\Lambda',$ according to the following lemma, will allow us to relate averages along primes with weighted averages over the integers.

\begin{lemma}[\cite{FHK2}]\label{L:n}
If $a:\N\to\C$ is bounded, then
$$
\Big| \frac{1}{\pi(N)}\sum_{p\in \P\cap[1,N]}a(p)-\frac{1}{N}\sum_{n=1}^N \Lambda'(n)\cdot a(n)\Big|=o_N(1).$$
\end{lemma}

Next, we recall the definition of Gowers norms.

\begin{definition*}
If $a:\Z_N\to\C,$ we inductively define:
$$
\norm{a}_{U_1(\Z_N)}=\Big|\E_{n\in\Z_N}a(n)\Big|;$$ and $$\norm{a}_{U_{d+1}(\Z_N)}=\left(\E_{h\in\Z_N}\norm{a_h\cdot \overline{a}}_{U_d(\Z_N)}^{2^d}\right)^{1/2^{d+1}},$$ where $a_h(n)=a(n+h).$ As Gowers showed in \cite{G}, $\norm{\cdot}_{U_{d}(\Z_N)}$ defines a norm on $\Z_N$ for $d\geq 2.$
\end{definition*}

For $w>2,$ let 
$$
W=\prod_{p\in\P\cap[1,w-1]}p$$ be the product of primes bounded above by $w.$ For $r\in \N,$ let 
$$
\Lambda_{w,r}'(n)=\frac{\phi(W)}{W}\cdot \Lambda'(Wn+r),$$ where $\phi$ is the Euler function, be the {\em modified von Mangoldt function}.

\medskip

 

The next result, which will play an important role in our proof, shows the Gowers uniformity of the modified von Mangoldt function and can be derived by \cite{GT}, \cite{GT2} and \cite{GTZ}. We write, and we will later use, the formulation that can be found in \cite{FHK}.

\begin{theorem}[\cite{GT},\cite{GT2},\cite{GTZ}]\label{T:g}
For every $d\in \N$ we have that 
$$
\lim_{w\to\infty}\left( \lim_{N\to\infty}\left( \max_{1\leq r\leq W,\;(r,W)=1}\norm{(\Lambda_{w,r}'-1)\cdot {\bf 1}_{[1,N]}}_{U_d(\Z_{dN})}\right) \right)=0.$$
\end{theorem}

The following uniformity estimate, is an important step towards our result. It will allow us to obtain a similar estimate in our setting, in Theorem~\ref{T:t6}, in order to make use of the conclusion of Theorem~\ref{T:g}.

\begin{lemma}[\mbox{\cite[Lemma 3.5]{FHK}}]\label{L:n1}
Let $\ell, m\in \N,$ $(X,\mathcal{X},\mu,T_1,\ldots, T_\ell)$ be a system, $q_{i,j}\in\Z[t]$ polynomials, $1\leq i\leq \ell,$ $1\leq j\leq m,$  $f_1,\ldots,f_m\in L^{\infty}(\mu)$  and $a:\N\to\C$ be a sequence satisfying $a(n)/n^c\to 0$ for every $c>0.$ Then there exists $d\in\N,$ depending only on the maximum degree of the polynomials $q_{i,j}$ and the integers $\ell,m$ and a constant $C_d$ depending on $d,$ such that 
$$
\norm{\frac{1}{N}\sum_{n=1}^N a(n)\cdot  (\prod_{i=1}^\ell T_i^{q_{i,1}(n)})f_1\cdot\ldots\cdot (\prod_{i=1}^\ell T_i^{q_{i,m}(n)})f_m}_{L^2(\mu)}\leq C_{d} \left( \norm{a\cdot {\bf 1}_{[1,N]}}_{U_d(\Z_{dN})}+o_N(1)\right). $$
Furthermore, the constant $C_d$ is independent of the sequence $(a(n))$ and the $o_N(1)$ term depends only on the integer $d$ and on the sequence $(a(n)).$ 
\end{lemma}

In order to prove this result, the authors in \cite{FHK}, successively made use of the following variant of van der Corput's estimate: 

\begin{lemma}\label{L:v}
Let $N\in \N$ and $v(1),\ldots,v(N)$ be elements in an inner product space with $\norm{v(i)}\leq 1.$ Then there exists a constant $C$ such that
$$
\norm{\frac{1}{N}\sum_{n=1}^N v(n)}^2\leq C\left( \frac{1}{N^2}\sum_{n=1}^N \norm{v(n)}^2 + \frac{1}{N}\sum_{h=1}^N\Big|\frac{1}{N}\sum_{n=1}^{N-h}\langle v(n+h),v(n) \rangle \Big|\right).
$$
\end{lemma}

\begin{remark*}
 In order to prove Lemma~\ref{L:n1}, we successively use Lemma~\ref{L:v}, using the van der Corput operation, choosing every time appropriate polynomials in order to have reduction in the complexity. Actually, the number $d$ that Lemma~\ref{L:n1} provides, corresponds to the $d-1$ steps we need to do in order the polynomials to be reduced into constant ones, by using the Polynomial Exhaustion Technique (PET) induction, introduced in \cite{Be87a}.
 
  For more information and details on the van der Corput operation and the scheme of the PET induction we are using here, we refer the reader to \cite{FHK}.
\end{remark*}



\section{Main estimates}

In this section we prove some key estimates needed in the proof of Theorems~\ref{T:e}, ~\ref{T:t3}, ~\ref{T:wg} and ~\ref{P:2}. The main idea is to use a transference principle that allows to deduce results for  integer part (and hence for closest integer) polynomials with real coefficients from respective results for polynomials with integer coefficients. 

\medskip

We first recall the definition of a measure preserving flow.

\begin{definition*}
Let $r\in \mathbb{N}$ and $(X,\mathcal{X}, \mu)$ be a probability space. 
We call a family $(T_t)_{t\in \mathbb{R}^r}$ of measure preserving 
transformations $T_t\colon X\to X,$ a {\em measure preserving flow}, if 
it satisfies $$T_{s+t}=T_s\circ T_t$$ for all $s,t\in \R^r.$
\end{definition*}

Also, we recall the notion of the {\em lower} and  {\em upper density} for a subset $S\subseteq \N$ to be the number $\underline{d}(S)$ and  $\overline{d}(S)$ respectively, where $$\underline{d}(S)=\liminf_{N\to\infty}\frac{|S\cap [1,N]|}{N}\;\;\;\text{and}\;\;\;\overline{d}(S)=\limsup_{N\to\infty}\frac{|S\cap [1,N]|}{N}.$$
If we have $\overline{d}(S)=\underline{d}(S),$ we call the common value {\em density} of $S.$

\medskip

The following important remarks contain essential information that we will use in various proofs along this article.

\begin{remarks*}
i) Let $\ell, m\in \N$ and $q_{i,j}\in\R[t]$ be real valued polynomials for $1\leq i\leq \ell,$ $1\leq j\leq m.$ Then,
for any $\R^{\ell m}$ measure preserving flow $\displaystyle \prod_{i=1}^{\ell}T_{i,s_{i,1}}\cdot \ldots\cdot \prod_{i=1}^\ell T_{i,s_{i,m}},$ where the transformations $T_{i,s_{i,j}}$ are defined in the probability space $(X,\mathcal{X},\mu),$ $s_{i,j}\in \R$ and functions $f_1,\ldots,$ $ f_m\in L^\infty(\mu),$  the sequence of functions $$\tilde{b}(n)= (\prod_{i=1}^\ell T_{i,q_{i,1}(n)})f_1\cdot \ldots\cdot (\prod_{i=1}^\ell T_{i,q_{i,m}(n)})f_m$$ satisfies the conclusion of Lemma~\ref{L:n1}, with $\tilde{b}(n)$ in place of the sequence $$(\prod_{i=1}^\ell T_{i}^{q_{i,1}(n)})f_1\cdot \ldots\cdot (\prod_{i=1}^\ell T_{i}^{q_{i,m}(n)})f_m.$$

Indeed, if $q(t)=a_r t^r+\ldots+a_1 t+a_0\in\R[t],$ we write $T_{q(t)}=(T_{a_r})^{n^r}\cdot\ldots\cdot(T_{a_1})^n\cdot T_{a_0}$ and we use Lemma~\ref{L:n1} for the invertible commuting measure preserving transformations $S_1=T_{a_1},\ldots,S_r=T_{a_r}.$ 

\medskip

\noindent ii) Real valued polynomials, $q\in\R[t],$ satisfy the condition:
$$\displaystyle \lim_{\delta\to 0^+} \overline{d}\left(\Big\{n\in\N\colon \{q(n)\}\in [1-\delta,1)\Big\}\right)=0,$$  where $\{\cdot\}$ denotes the fractional part.
  
\medskip  
  
  Indeed, let $q(t)=a_r t^r+\ldots+a_1 t+a_0\in \R[t].$

If $a_i\notin \Q$ for some $1\leq i\leq r,$ then we have the condition from Weyl's result, since $(q(n))$ is uniformly distributed (mod 1).

If $a_i\in \mathbb{Q}$ for all $1\leq i\leq r,$ then the sequence $(q(n))$ is periodic (mod 1) and the
conclusion is obvious.

\medskip

\noindent iii) If $f$ is Riemann-integrable on $[0,1)$ with $\int_{[0,1)}f(x)\,dx=c$,  then  for every $\varepsilon>0$ we can find trigonometric polynomials $q_1,\;q_2,$ with no constant terms, in order to have the relation $$q_1(t)+c-\varepsilon\leq f(t)\leq q_2(t)+c+\varepsilon.$$
\end{remarks*}





The following result is proved via a transference principle that enables one to deduce results for $\Z$-actions from results for flows (see \cite{K} in comparison with \cite{F}).  This technique was first used in \cite{EL} by E.~Lesigne, in order to prove that when a sequence of real positive numbers is good for the single term pointwise ergodic theorem, then the respective sequence of its integer parts is also good 
(see also \cite{BJW}). This method was later used by M.~Wierdl (in \cite{Mate})  to deal with multiple term averages (see Theorem~3.2 in \cite{K}).


\begin{theorem}\label{T:t6}
Let $\ell, m\in \N,$ $(X,\mathcal{X},\mu,T_1,\ldots, T_\ell)$ be a system, $q_{i,j}\in\R[t]$ polynomials, $1\leq i\leq \ell,$ $1\leq j\leq m$ and  $f_1,\ldots,f_\ell\in L^{\infty}(\mu).$

For the sequence of functions $$b(n)= (\prod_{i=1}^\ell T_i^{[q_{i,1}(n)]})f_1\cdot\ldots\cdot(\prod_{i=1}^\ell T_i^{[q_{i,m}(n)]})f_m$$ there exists $d\in \N,$ depending only on the maximum degree of the polynomials $q_{i,j}$ and the integers $\ell$ and $m,$ such that for every $0<\delta<1$ there exists a constant $C_{d,\delta}$ depending on $d$ and $\delta,$ such that 
\begin{eqnarray*}
\norm{\frac{1}{N}\sum_{n=1}^N(\Lambda'_{w,r}(n)-1)b(n)}_{L^2(\mu)}& \leq & C_{d,\delta}\left(\norm{(\Lambda'_{w,r}-1)\cdot {\bf 1}_{[1,N]}}_{U_d(\Z_{dN})}+o_N(1)\right) \\ 
& +& c_{\delta}\left(1+o_{N\to\infty;w}(1)\right),
\end{eqnarray*}
for all $r\in \N,$ where  $c_\delta\to 0$ as $\delta\to 0^+$ and the term $o_N(1)$ depends on the integer $d.$
\end{theorem}

\begin{remark*}
Due to unpleasant error terms, it seems difficult to adapt the PET induction used in the proof of Lemma~\ref{L:n1} in order to prove the asserted estimate.
\end{remark*}

\begin{proof}[Proof of Theorem~\ref{T:t6}]
Let $0<\delta<1$ and $w, r\in \N.$
For the given transformations on $X,$ we define the $\R^{\ell m}$ action $\displaystyle \prod_{i=1}^{\ell}T_{i,s_{i,1}}\cdot \ldots\cdot \prod_{i=1}^\ell T_{i,s_{i,m}}$  on the probability space $Y=X\times [0,1)^{\ell m},$ endowed with the measure $\nu=\mu\times \lambda^{\ell m}$ ($\lambda$ is the Lebesgue measure on $[0,1)$), by
$$\prod_{j=1}^m \prod_{i=1}^\ell T_{i,s_{i,j}}(x,a_{1,1},\ldots,a_{\ell,1},a_{1,2},\ldots,a_{\ell,2},\ldots,a_{1,m},\ldots,a_{\ell,m})=$$
$$\left(\prod_{j=1}^m \prod_{i=1}^\ell T_i^{[s_{i,j}+a_{i,j}]}x,\{s_{1,1}+a_{1,1}\},\ldots,\{s_{\ell,1}+a_{\ell,1}\},\ldots,\{s_{1,m}+a_{1,m}\},\ldots,\{s_{\ell,m}+a_{\ell,m}\}\right). $$
Since the transformations $T_1,\ldots, T_\ell$ are measure preserving and commute, and also since $[x+\{y\}]+[y]=[x+y],$ it is easy to check that the above action defines a measure preserving flow on the product probability space $Y$.

 If $f_1,\ldots,f_m$ are bounded functions on $X,$ for every element
 
 $(a_{1,1},\ldots,a_{\ell,1},a_{1,2},\ldots,a_{\ell,2},\ldots,a_{1,m},\ldots,a_{\ell,m})\in [0,1)^{\ell m}$   we define the $Y$-extensions of $f_j,$ setting:  $$\hat{f}_j(x,a_{1,1},\ldots,a_{\ell,1},a_{1,2},\ldots,a_{\ell,2},\ldots,a_{1,m},\ldots,a_{\ell,m})=f_j(x),\;\;1\leq j\leq m;$$ and $$\hat{f}_0(x,a_{1,1},\ldots,a_{\ell,1},a_{1,2},\ldots,a_{\ell,m})= 1_{[0,\delta]^{\ell m}}(a_{1,1},\ldots,a_{\ell,1},a_{1,2},\ldots,a_{\ell,m}).$$ If $$\tilde{b}(n)= \hat{f}_0\cdot(\prod_{i=1}^\ell T_{i,q_{i,1}(n)})\hat{f}_1\cdot \ldots\cdot (\prod_{i=1}^\ell T_{i,q_{i,m}(n)})\hat{f}_m,$$  for every $x\in X$ we define $$b'(n)(x)=\int_{[0,1)^{\ell m}}\tilde{b}(n)(x,a_{1,1},\ldots,a_{\ell,1},a_{1,2},\ldots,a_{\ell,m})\,d\lambda^{\ell m},$$ where the integration is with respect to the variables $a_{i,j}.$

Then, by using the triangle and the Cauchy-Schwarz inequality, for the sequence $a(n)=\Lambda'_{w,r}(n)-1,$ we have that 
$$\delta^{\ell m}\norm{\frac{1}{N}\sum_{n=1}^N a(n)b(n)}_{L^2(\mu)}\leq \norm{\frac{1}{N}\sum_{n=1}^N a(n)\cdot (\delta^{\ell m}b(n)-b'(n))}_{L^2(\mu)}+\norm{\frac{1}{N}\sum_{n=1}^N a(n)\tilde{b}(n)}_{L^2(\nu)}.$$
From Part~i) of the previous remark, we can use Lemma~\ref{L:n1} to find an integer $d\in \N,$ depending only on the maximum degree of the polynomials $q_{i,j}$ and the integers $\ell,m$ and a constant $C_d$ depending on $d,$  such that
$$\norm{\frac{1}{N}\sum_{n=1}^N a(n)\tilde{b}(n)}_{L^2(\nu)}\leq C_d\left(\norm{a\cdot {\bf 1}_{[1,N]}}_{U_d(\Z_{dN})}+o_N(1)\right),$$
where the $o_N(1)$ term depends only on the integer $d$ and the sequence $(a(n)).$

Next we will study the term $\norm{\frac{1}{N}\sum_{n=1}^N a(n)\cdot (\delta^{\ell m}b(n)-b'(n))}_{L^2(\mu)}.$

For every $x\in X$ we have
$$\Big|\delta^{\ell m} b(n)(x)-b'(n)(x)\Big|= \Big|\int_{[0,\delta]^{\ell m}}  \left(\prod_{j=1}^m f_j(\prod_{i=1}^\ell T_i^{[q_{i,j}(n)]}x)-\prod_{j=1}^m f_j(\prod_{i=1}^\ell T_i^{[q_{i,j}(n)+a_{i,j}]}x) \right)\, d \lambda^{\ell m}\Big|. $$
Since all the relevant $a_{i,j}$ in the integrand are less or equal than $\delta,$ if the fractional part of all $q_{i,j}(n)$ is less than $1-\delta,$ we have $T_i^{[q_{i,j}(n)+a_{i,j}]}=T_i^{[q_{i,j}(n)]}$ for all $1\leq i\leq \ell,$ $1\leq j\leq m.$ We will deal with the case where the fractional part of some $q_{i,j}(n)$ is greater or equal than $1-\delta.$ 

For every $1\leq i\leq \ell,$ $1\leq j\leq m,$ let
$$E_{\delta}^{i,j}:=\{n\in \N\colon \{q_{i,j}(n)\}\in [1-\delta,1)\}.$$
Then, by using the fact that ${\bf 1}_{E_\delta^{1,1}\cup \ldots\cup E_\delta^{1,m}\cup E_\delta^{2,1}\cup\ldots\cup E_\delta^{\ell,m}}\leq \sum_{(i,j)\in[1,\ell]\times[1,m]} {\bf 1}_{E_{\delta}^{i,j}} $ and that ${\bf 1}_{E_{\delta}^{i,j}}(n)={\bf 1}_{[1-\delta,1)}(\{q_{i,j}(n)\}),$ for $1\leq i\leq \ell,$ $1\leq j\leq m,$ $n\in\N,$  for every $x\in X$ we have
$$\Big|\delta^{\ell m}b(n)(x)-b'(n)(x)\Big|\leq 2\delta^{\ell m}\cdot \sum_{(i,j)\in[1,\ell]\times[1,m]} \frac{1}{N}\sum_{n=1}^N  {\bf 1}_{[1-\delta,1)}(\{q_{i,j}(n)\}) ,$$ so, recalling that $a(n)=\Lambda_{w,r}'(n)-1,$
$$ \frac{1}{N}\sum_{n=1}^N |a(n)|\cdot {\bf 1}_{[1-\delta,1)}(\{q_{i,j}(n)\}) \leq  \frac{1}{N}\sum_{n=1}^N \Lambda'_{w,r}(n)\cdot{\bf 1}_{[1-\delta,1)}(\{q_{i,j}(n)\})+\frac{|E_{\delta}^{i,j}\cap [1,N]|}{N}.$$
From Part ii) of the previous remark, we have that for small enough $\delta,$ the term (and the sum of finitely many terms of this form)  $\frac{|E_{\delta}^{i,j}\cap [1,N]|}{N}$ is as small as we want.

As for the term $\frac{1}{N}\sum_{n=1}^N \Lambda'_{w,r}(n)\cdot{\bf 1}_{[1-\delta,1)}(\{q_{i,j}(n)\}),$ we will show that it goes to zero when $N\to\infty,$ then $w\to\infty$ and finally $\delta\to 0^+.$

 If the polynomial $q_{i,j}(n)$ has only (except maybe the constant term) rational coefficients, for small $\delta,$ the sum will go to zero from periodicity, and so, we can assume that the polynomial has at least one irrational coefficient (except the constant term). From Part iii) of the previous remark (for the function $f={\bf 1}_{[1-\delta,1)}$), it suffices to show that $\frac{1}{N}\sum_{n=1}^N \Lambda'_{w,r}(n)e^{2\pi i k q_{i,j}(n)}\to 0$ as $N\to\infty$ and then $w\to\infty$ for all $k\in \Z\setminus\{0\}.$
 
We write 
$$\frac{1}{N}\sum_{n=1}^N \Lambda'_{w,r}(n)e^{2\pi i k q_{i,j}(n)}=\frac{1}{N}\sum_{n=1}^N (\Lambda'_{w,r}(n)-1)e^{2\pi i k q_{i,j}(n)}+\frac{1}{N}\sum_{n=1}^N e^{2\pi i k q_{i,j}(n)}.$$ The first term goes to zero  as $N\to\infty$ and then $w\to\infty,$ from \cite{GT},  since $(e^{2\pi i k q_{i,j}(n)})_n$ is a nilsequence, while the second term goes to zero as $N\to\infty$ from Weyl's equidistribution theorem. The result now follows.
\end{proof}

Using first Theorem~\ref{T:t6} and then Theorem~\ref{T:g}, we will now prove a result, in Proposition~\ref{P:1} below, which will give us a comparison between averages over primes (via the modified von Mangoldt function) and averages over integers. This result, together with a uniform multiple recurrence result, that we prove in Corollary~\ref{C:n2} below, reflect the main arguments for proving  Theorem~\ref{T:e}. The proof of Theorem~\ref{T:e} will actually make use of the Proposition~\ref{P:1} for the respective closest integer polynomial iterates. Corollary~\ref{C:n2} will mainly follow by a uniform multiple recurrence result for polynomial iterates, which follows from the polynomial Szemer{\'e}di theorem (from \cite{BL2}) in the same way as Theorem~2.1~(ii) and (iii) is proved in \cite{BHMP}, in order to obtain the respective uniform multiple recurrence result for closest integer polynomial iterates, in Proposition~\ref{P:102}, the proof of which is using a method presented in \cite{BHa} and used in \cite{WS} as well.

\begin{proposition}\label{P:1}
Let $\ell, m\in \N,$ $(X,\mathcal{X},\mu,T_1,\ldots, T_\ell)$ be a system, $q_{i,j}\in\R[t]$ polynomials, $1\leq i\leq \ell,$ $1\leq j\leq m$ and  $f_1,\ldots,f_\ell\in L^{\infty}(\mu).$

Then,
$$
 \max_{1\leq r\leq W,\;(r,W)=1}\norm{\frac{1}{N}\sum_{n=1}^N(\Lambda_{w,r}'(n)-1)\cdot (\prod_{i=1}^\ell T_i^{[q_{i,1}(Wn+r)]})f_1\cdot\ldots\cdot(\prod_{i=1}^\ell T_i^{[q_{i,m}(Wn+r)]})f_m}_{L^2(\mu)}
$$
converges to $0$ as $N\to\infty$ and then $w\to\infty.$
\end{proposition}

\begin{proof}
Using Theorem~\ref{T:t6}, for the polynomials $q_{i,j}(Wn+r),$ we get that for every $0<\delta<1,$ there exists $d\in\N,$ depending only on the maximum degree of the polynomials $q_{i,j}$ and the integers $\ell$ and $m,$ and a constant $C_{d,\delta}$ depending on $d$ and $\delta,$ such that
$$
 \max_{1\leq r\leq W,\;(r,W)=1}\norm{\frac{1}{N}\sum_{n=1}^N(\Lambda_{w,r}'(n)-1)\cdot (\prod_{i=1}^\ell T_i^{[q_{i,1}(Wn+r)]})f_1\cdot\ldots\cdot(\prod_{i=1}^\ell T_i^{[q_{i,m}(Wn+r)]})f_m}_{L^2(\mu)}
$$
$$
\leq C_{d,\delta}\left(\max_{1\leq r\leq W,\;(r,W)=1}\norm{(\Lambda'_{w,r}-1)\cdot {\bf 1}_{[1,N]}}_{U_d(\Z_{dN})}+o_N(1)\right)+c_{\delta}\left(1+o_{N\to\infty;w}(1)\right),$$
where $c_\delta\to 0$ as $\delta\to 0^+.$ Taking first $N\to\infty$ and then $w\to\infty$ in this expression, by Theorem~\ref{T:g}, we have that the required limit is bounded above by $c_\delta.$ Taking $\delta\to 0^+,$ we get the result.
\end{proof}

We will make use of the following uniform multiple recurrence result; it follows from the polynomial Szemer{\'e}di (\cite{BL2}) in the same way as Theorem~2.1~(ii) is proved in \cite{BHMP}: 

\begin{theorem}[\mbox{\cite{BHMP}}]\label{C:n1}
Let $\ell, m\in \N$ and $(X,\mathcal{X},\mu,(T_{i,j})_{1\leq i\leq \ell, 1\leq j\leq m})$ be a system. Then for any $A\in \mathcal{X}$ with $\mu(A)>0,$ there exists a positive constant $c\equiv c_{\ell,m,\mu(A)}>0$  such that   
\begin{equation}\label{E:n1e}
\liminf_{N\to\infty}\frac{1}{N}\sum_{n=1}^N\mu\left( A\cap (\prod_{i=1}^\ell T_{i,1}^{-n^i})A\cap\ldots\cap (\prod_{i=1}^\ell T_{i,m}^{-n^i})A\right)\geq c.
\end{equation}
\end{theorem}

\begin{remark*}
In fact, it is known that the limit in \eqref{E:n1e} exists from \cite{W12}.
\end{remark*}





\begin{lemma}\label{L:101}
 Let $\ell, m\in \N.$ For every $\delta>0$ there exists a constant $c\equiv c_{\ell,m,\delta}>0$ such that:
$$\underline{d}\left(\Big\{n\in\N:\;\norm{a_{i,j}n^i}<\delta\;\;\forall\;1\leq i\leq \ell,\;1\leq j\leq m\Big\}\right)\geq c,$$ for all $a_{i,j}\in \R,$ $1\leq i\leq \ell,$ $1\leq j\leq m.$
\end{lemma}

\begin{proof}
Let $0<\delta\leq 1/2.$ For the system $(\T,\mathcal{B}(\T),m,(T_{i,j})_{1\leq i\leq \ell, 1\leq j\leq m}),$ where $$T_{i,j}x=x+a_{i,j}(mod 1),\;\;x\in\T,\;\;1\leq i\leq \ell,\;1\leq j\leq m$$ and the set $A=[0,\delta),$ use Theorem~\ref{C:n1} to find a constant  $c\equiv c_{\ell,m,\delta}>0$ so that
$$\liminf_{N\to\infty}\frac{1}{N}\sum_{n=1}^N m\left(A\cap(\prod_{i=1}^\ell T_{i,1}^{-n^i})A\cap\ldots\cap (\prod_{i=1}^\ell T_{i,m}^{-n^i})A \right)\geq c.$$
To obtain the result, note that if $x\in A\cap(\prod_{i=1}^\ell T_{i,1}^{-n^i})A\cap\ldots\cap (\prod_{i=1}^\ell T_{i,m}^{-n^i})A,$ then
 \begin{equation}\label{E:101}
x+a_{i,j}n^i (mod 1)\in [0,\delta),\;\;\forall\;1\leq i\leq \ell,\;1\leq j\leq m,
\end{equation}
 and since $x\in [0,\delta)$ itself,  \eqref{E:101} gives us that $$\{a_{i,j}n^i\}\in [0,\delta-x)\cup [1-x,1),\;\;\forall\;1\leq i\leq \ell,\;1\leq j\leq m,$$ from which we have that $$\norm{a_{i,j}n^i}<\delta\;\;\forall\;1\leq i\leq \ell,\;1\leq j\leq m,$$ hence the result.
\end{proof}

\begin{remark*}
The proof of Lemma~\ref{L:101} can also follow by a single recurrence argument as well.
\end{remark*}

We will also need the following lemma.

\begin{lemma}\label{L:102}
Let $q(n)=a_1 n+\ldots+a_k n^k \in \R[n],$ be a real valued polynomial with no constant term and for any $r\in \N$ let $$ S_r:=\Big\{m\in \N:\;\norm{a_i m^i}<\frac{1}{2k r^k},\;1\leq i\leq k\Big\}.$$ Then, for any $m\in S_r$ and $1\leq n\leq r$ we have $$[[q(mn)]]=[[a_1 m]]n+\ldots+[[a_k m^k]]n^k.$$
\end{lemma}

\begin{proof}
If $\norm{x_1}+\ldots+\norm{x_k}<1/2,$ we can easily check that $[[x_1+\ldots+x_k]]=[[x_1]]+\ldots+[[x_k]],$  obtaining the conclusion of the lemma.
\end{proof}

We will also use the following uniform multiple recurrence result which follows from the polynomial Szemer{\'e}di (\cite{BL2}) in the same way as Theorem~2.1~(iii) is proved in \cite{BHMP}: 

\begin{theorem}[\mbox{\cite{BHMP}}]\label{C:n11}
Let $\ell, m\in \N$ and $(X,\mathcal{X},\mu,(T_{i,j})_{1\leq i\leq \ell, 1\leq j\leq m})$ be a system. Then for any $A\in \mathcal{X}$ with $\mu(A)>0,$ there exist a constant $c\equiv c_{\ell,m,\mu(A)}>0$ and an integer $N_0\equiv N_0(\ell,m,\mu(A))\in\N$  such that for some $1\leq n\leq N_0$ we have  
\begin{equation}\label{E:n1e1}
\mu\left( A\cap (\prod_{i=1}^\ell T_{i,1}^{-n^i})A\cap\ldots\cap (\prod_{i=1}^\ell T_{i,m}^{-n^i})A\right)\geq c.
\end{equation}
\end{theorem}

Using Theorem~\ref{C:n11}, we get the following:

\begin{proposition}\label{P:1011}
Let $\ell, m\in \N,$ $(X,\mathcal{X},\mu,T_1,\ldots,T_\ell)$ be a system and $q_{i,j}\in \Z[t]$ polynomials with maximum degree $d$ and $q_{i,j}(0)=0$ for any $1\leq i\leq \ell, 1\leq j\leq m.$ Then, for any $A\in \mathcal{X}$ with $\mu(A)>0$ there exist a constant $c\equiv c_{d,m,\mu(A)}>0$ and an integer $N_0\equiv N_0(\ell,m,\mu(A))\in\N$  such that for some $1\leq n\leq N_0$ we have
\begin{equation}\label{E:nlo1}
\mu\left(A\cap(\prod_{i=1}^\ell T_{i}^{-q_{i,1}(n)})A\cap\ldots\cap (\prod_{i=1}^\ell T_{i}^{-q_{i,m}(n)})A \right)\geq c.
\end{equation}
\end{proposition}

\begin{proof}
If $q_{i,j}(n)=a_{1,i,j}n+\ldots+a_{d,i,j}n^d$ for any $1\leq i\leq \ell, 1\leq j\leq m$ (we put some zero terms if needed), set
$$S_{k,j}=\prod_{i=1}^\ell T_i^{a_{k,i,j}},\;\;1\leq k\leq d,\;1\leq j\leq m,$$ and use Theorem~\ref{C:n11} for the system $(X,\mathcal{X},\mu,(S_{k,j})_{1\leq k\leq d, 1\leq j\leq m}).$
\end{proof}

\begin{proposition}\label{P:102}
Let $\ell, m\in \N,$ $(X,\mathcal{X},\mu,T_1,\ldots,T_\ell)$ a system and $q_{i,j}\in \R[t]$ with maximum degree $d$ and $q_{i,j}(0)=0$ for any $1\leq i\leq \ell, 1\leq j\leq m.$ Then, for any $A\in \mathcal{X}$ with $\mu(A)>0$ there exists a constant $c\equiv c_{d,\ell,m,\mu(A)}>0$  such that  
\begin{equation}\label{E:nnn}
\liminf_{N\to\infty}\frac{1}{N}\sum_{n=1}^N\mu\left(A\cap(\prod_{i=1}^\ell T_{i}^{-[[q_{i,1}(n)]]})A\cap\ldots\cap (\prod_{i=1}^\ell T_{i}^{-[[q_{i,m}(n)]]})A \right)\geq c.
\end{equation}
\end{proposition}

\begin{proof} Write $q_{i,j}(n)=a_{1,i,j}n+\ldots+a_{d,i,j}n^d,\;\;1\leq i\leq \ell,\;1\leq j\leq m$ (by putting some zero terms if needed). 
For any $r\in \N,$ let $$S_r=\Big\{s\in\N:\;\norm{a_{k,i,j}s^k}<\frac{1}{2d r^{d}}\;\;\forall\;1\leq k\leq d,\;1\leq i\leq \ell,\;1\leq j\leq m\Big\}.$$

\noindent Use Lemma~\ref{L:101} to bound the lower density of $S_r$  below by some constant $c_{d,\ell,m,r}>0.$

\medskip

For any $s\in S_r,\;1\leq n\leq r,$ using Lemma~\ref{L:102}, we have 

\begin{equation*}
[[q_{i,j}(ns)]]  = [[a_{1,i,j} s]]n+\ldots+[[a_{d,i,j} s^d]]n^d,\;\;1\leq i\leq \ell, \;1\leq j\leq m.
\end{equation*} Set $$p_{s,i,j}(n)=[[a_{1,i,j} s]]n+\ldots+[[a_{d,i,j} s^d]]n^d,\;\;s\in \N,\;1\leq i\leq \ell, \;1\leq j\leq m.$$

\noindent Use Proposition~\ref{P:1011} to find a positive constant $\tilde{c}\equiv\tilde{c}_{d,\mu(A),\ell,m}$ and a positive integer $N_0\equiv N_0(d,\mu(A),\ell,m)$ such that, for some $1\leq n\leq N_0,$ for any $s\in\N,$ we have 
$$\mu\left(A\cap(\prod_{i=1}^\ell T_{i}^{-p_{s,i,1}(n)})A\cap\ldots\cap (\prod_{i=1}^\ell T_{i}^{-p_{s,i,m}(n)})A \right)\geq \tilde{c}. $$
Then, 
$$\liminf_{N\to \infty}\frac{1}{N}\sum_{n=1}^N\mu\left(A\cap(\prod_{i=1}^\ell T_{i}^{-[[q_{i,1}(n)]]})A\cap\ldots\cap (\prod_{i=1}^\ell T_{i}^{-[[q_{i,m}(n)]]})A \right)\geq$$ $$ \frac{1}{N_0}\liminf_{N\to\infty}\frac{1}{N}\sum_{s=1}^{[\frac{N}{N_0}]}\sum_{n=1}^{N_0}\mu\left(A\cap(\prod_{i=1}^\ell T_{i}^{-[[q_{i,1}(ns)]]})A\cap\ldots\cap (\prod_{i=1}^\ell T_{i}^{-[[q_{i,m}(ns)]]})A \right)\geq$$
$$\frac{1}{N_0}\liminf_{N\to\infty}\frac{1}{N}\sum_{s\in S_{N_0}\cap\{1,\ldots,[\frac{N}{N_0}]\}}\sum_{n=1}^{N_0}\mu\left(A\cap(\prod_{i=1}^\ell T_{i}^{-p_{s,i,1}(n)})A\cap\ldots\cap (\prod_{i=1}^\ell T_{i}^{-p_{s,i,m}(n)})A \right)\geq $$
 $$\frac{\tilde{c}_{d,\mu(A),\ell,m}}{N_0^2}\liminf_{N\to\infty}\frac{\Big| S_{N_0}\cap\Big\{1,\ldots,\Big[\frac{N}{N_0}\Big]\Big\} \Big|}{[\frac{N}{N_0}]}\geq \frac{\tilde{c}_{d,\mu(A),\ell,m}\cdot c_{d,N_0,\ell,m}}{N_0^2}>0, $$ and we have the result.
\end{proof}

By using Proposition~\ref{P:102} we immediately obtain the following uniform result, the second ingredient we need 
in order to prove Theorem~\ref{T:e}. The uniformity in $W$ is crucial for the proof of  Theorem~\ref{T:e}.

\begin{corollary}\label{C:n2}
Let $\ell, m \in \N,$  $(X,\mathcal{X},\mu,T_1,\ldots, T_\ell)$ be a system and $q_{i,j}\in \R[t]$ be polynomials with maximum degree $d$ and $q_{i,j}(0)=0$ for $1\leq i\leq \ell,$ $1\leq j\leq m.$ Then for any $A\in \mathcal{X}$ with $\mu(A)>0,$ there exists a constant $c\equiv c_{d,\ell,m,\mu(A)}>0$ such that for every $W\in \N$
\begin{equation}\label{E:n2c}
\liminf_{N\to\infty}\frac{1}{N}\sum_{n=1}^N\mu \left( A\cap (\prod_{i=1}^\ell T_i^{-[[q_{i,1}(Wn)]]})A\cap\ldots\cap (\prod_{i=1}^\ell T_i^{-[[q_{i,m}(Wn)]]})A\right)\geq c.
\end{equation}
\end{corollary}

\begin{remark*}
In fact, it is known that the limits in \eqref{E:nnn} and \eqref{E:n2c} exist from \cite{K}.
\end{remark*}

\section{Proof of main results}

In this section we give the proof of Theorems~\ref{T:e},~\ref{T:t3} and \ref{T:wg}. First, we prove Theorem~\ref{T:e} for the $\P-1$ case (the $\P+1$ case follows similarly).

\begin{proof}[Proof of Theorem~\ref{T:e}]
Using Proposition~\ref{P:1} for the polynomials $q_{i,j}(n-1)+\frac{1}{2}$ with $r=1$ and Corollary~\ref{C:n2}, we have (recall that $[[x]]=[x+1/2]),$ for sufficiently large $w\in \N,$ that
$$
\liminf_{N\to\infty}\frac{1}{N}\sum_{n=1}^N \Lambda_{w,1}'(n)\cdot\mu\left(A\cap (\prod_{i=1}^\ell T_i^{-[[q_{i,1}(Wn)]]})A\cap\ldots\cap (\prod_{i=1}^\ell T_i^{-[[q_{i,m}(Wn)]]})A\right)>0,
$$ 
from which we have the required non-empty intersection with $\P-1.$
\end{proof}

\begin{remark*}
According to Lemma~\ref{L:n}, we have that the conclusion of Theorem~\ref{T:e}, and so of Theorem~\ref{T:p}, is satisfied for a set of integers $n$ with positive relative density in the shifted primes $\P-1.$ The analogous result, by a similar argument, holds for the set $\P+1$  as well.
\end{remark*}

\begin{proof}[Proof of Theorem~\ref{T:t3}]
We borrow the arguments from the proof of Theorem~1.3 from \cite{FHK}. By Lemma~\ref{L:n} it suffices to show that the following sequence is Cauchy in $L^2(\mu):$
$$
A(N):=\frac{1}{N}\sum_{n=1}^N \Lambda'(n)\cdot (\prod_{i=1}^\ell T_i^{[q_{i,1}(n)]})f_1\cdot\ldots\cdot (\prod_{i=1}^\ell T_i^{[q_{i,m}(n)]})f_m.
$$
Using Proposition~\ref{P:1}, for any $\varepsilon>0,$ if for $w,r\in\N,$ we define   $$B_{w,r}(N):=\frac{1}{N}\sum_{n=1}^N   (\prod_{i=1}^\ell T_i^{[q_{i,1}(Wn+r)]})f_1\cdot\ldots\cdot (\prod_{i=1}^\ell T_i^{[q_{i,m}(Wn+r)]})f_m.$$  Then for sufficiently large $N$ and some $w_0$ (which gives us a corresponding $W_0$) we have 
\begin{equation}\label{E:ff}
\norm{A(W_0 N)-\frac{1}{\phi(W_0)}\sum_{1\leq r\leq W_0,\;(r,W_0)=1}B_{w_0,r}(N)}_{L^2(\mu)}<\varepsilon.
\end{equation}
Using the fact that for any $1\leq r\leq W_0$ the sequence $(B_{w_0,r}(N))$ is Cauchy in $L^2(\mu),$ which follows from \cite{K}, as well as the Relation~\eqref{E:ff}, we get that for $M,N$ sufficiently large $$\norm{A(W_0 M)-A(W_0 N)}_{L^2(\mu)}<\varepsilon.$$ Then $(A(N))$ is Cauchy in $L^2(\mu),$  since $$\norm{A(W_0 N+r)-A(W_0 N)}_{L^2(\mu)}=o_N(1),$$ for every $1\leq r\leq W_0.$
\end{proof}

\begin{proof}[Proof of Theorem~\ref{T:wg}]
We use the same argument as in the proof of Theorem~\ref{T:e}. The only difference is that we need to prove a variant of Corollary~\ref{C:n2} for the integer part of the polynomials  $q_{i,j}(n)=ak_{i,j}n^{d_{i,j}},\;1\leq i\leq \ell,$ $1\leq j\leq m.$ Let $d=\max\{d_{i,j}:\; 1\leq i\leq \ell,$ $ 1\leq j\leq m\},$ $k=\max\{k_{i,j}:\; 1\leq i\leq \ell, 1\leq j\leq m\},$ 
 $W\in \N$ and, for $r\in \N,$ set $$S_r=\Big\{m\in \N:\;\{aW^{d_{i,j}}m^{d_{i,j}}\}<\frac{1}{kr^d},\;1\leq i\leq \ell, 1\leq j\leq m\Big\}$$ (without loss of generality $k\in\N$). It is sufficient to find a positive lower bound for the lower density of $S_r$ independent of $W.$ Then, since, for any $m\in S_r$ and $1\leq n\leq r,$ we have $$[q_{i,j}(Wmn)]=[ak_{i,j}W^{d_{i,j}}m^{d_{i,j}}n^{d_{i,j}}]=[ak_{i,j}W^{d_{i,j}}m^{d_{i,j}}]n^{d_{i,j}},$$  we can follow the proof of Proposition~\ref{P:102} to obtain the result.

 If $a=t/s\in \Q$ with $s\in \N,$ then $\underline{d}(S_r)\geq 1/s.$
 
  If $a\notin \Q,$ then, if $\{d_{i,j}:\;1\leq i\leq \ell, \;1\leq j\leq m\}=\{d_1<\ldots<d_\xi=d\},$ since $(a(nW)^{d_1},\ldots,a(nW)^{d_\xi})_n$ is equidistributed in $\T^\xi,$ we get $d(S_r)=\frac{1}{(kr^{d})^\xi}\geq \frac{1}{(kr^{d})^{\ell m }}.$
\end{proof}

\section{An application on Gowers uniform sets}

In this last section, following \cite{FranH}, we will give an application of our approach, on Gowers uniform sets. More specifically, we will prove Theorem~\ref{P:2}, i.e., any shift of a Gowers uniform set is a set of closest integer polynomial multiple recurrence and of integer part polynomial multiple mean convergence (see definitions below). The main ingredients in order to obtain this result will be Theorem~\ref{T:ref} and a result from \cite{FranH}, Lemma~\ref{L:nh}. 

\medskip

Imitating \cite{FranH}, we give the following definitions:

\begin{definition*}
i) A set of integers
$S$ is {\em a set of integer part polynomial multiple mean convergence} if for every $\ell, m\in \N,$ system 
$(X, \mathcal{X} , \mu, T_1,\ldots,$ $ T_\ell),$
functions $f_1,\ldots, f_m \in L^\infty(\mu)$ and polynomials $q_{i,j}\in \R[t]$ for $1\leq i\leq \ell,$ $1\leq j\leq m,$
the averages 
$$
\frac{1}{|S\cap [1,N]|}\sum_{n\in S\cap [1,N]} (\prod_{i=1}^\ell T_i^{[q_{i,1}(n)]})f_1\cdot\ldots\cdot (\prod_{i=1}^\ell T_i^{[q_{i,m}(n)]})f_m
$$
converge in $L^2(\mu)$ as $N\to\infty.$

\medskip

ii) A set of integers $S$ is {\em a set of closest integer polynomial multiple recurrence} if for every $\ell, m\in \N,$
system $(X, \mathcal{X} , \mu, T_1,\ldots, T_\ell),$ set $A\in \mathcal{X}$ with $\mu(A)> 0$ and polynomials $q_{i,j}\in \R[t]$ with $q_{i,j} (0) = 0,$ $1\leq i\leq \ell,$ $1\leq j\leq m,$ we have
$$
\mu\left(A\cap(\prod_{i=1}^\ell T_i^{-[[q_{i,1}(n)]]})A\cap\ldots\cap (\prod_{i=1}^\ell T_i^{-[[q_{i,m}(n)]]})A\right)>0
$$
for a set of $n\in S$ with positive lower relative, in $S,$ density.
\end{definition*} 

\begin{remark*}
The aforementioned definitions are equivalent to the ones in \cite{FranH} for sets of positive density in $\N.$
\end{remark*}

The following result is a corollary of Theorems~\ref{T:e} and ~\ref{T:t3}.

\begin{theorem}\label{T:cf2}
The set of shifted primes $\P-1$ (and similarly $\P+1$) is a set of closest integer polynomial multiple recurrence and the set of primes, $\P,$ is a set of integer part polynomial multiple mean convergence.
\end{theorem}

We will now prove that $\N$ is a set of closest integer polynomial multiple recurrence and of integer part polynomial multiple mean convergence,  which we will use in the proof of Theorem~\ref{P:2}.

\begin{theorem}\label{T:ref}
The set of positive integers, $\N,$ is a set of closest integer polynomial multiple recurrence and of integer part polynomial multiple mean convergence.
\end{theorem}

\begin{proof}
 That $\N$ is a set of closest integer polynomial multiple recurrence follows from Proposition~\ref{P:102}.    The integer part polynomial multiple mean convergence of $\N$ follows from Theorem~1.4 of \cite{K}.
\end{proof}



In order to prove Theorem~\ref{P:2}, we also need the following  result of Frantzikinakis and Host, which we borrow from \cite{FranH}:

\begin{lemma}[\mbox{\cite[Lemma~2.4]{FranH}}]\label{L:nh}
Let $d\geq 2$ be an integer and $\varepsilon,\kappa>0.$ Then there exist $\eta>0$ and $N_0\in \N,$ such that for all integers $N,\tilde{N}$ with $N_0\leq N\leq \tilde{N}\leq \kappa N,$ every interval $J\subseteq [1,N]$ and $g:\Z_{\tilde{N}}\to\C$ with $|g|\leq 1,$ the following implication holds:
\begin{center} if $\;\norm{g}_{U_d(\Z_{\tilde{N}})}\leq\eta,\;$ then $\;\norm{g\cdot{\bf 1}_J}_{U_d(\Z_{N})}\leq \varepsilon.$ \end{center}
\end{lemma}

We will now recall the definition of a Gowers uniform set.

\begin{definition*}[\cite{FranH}]
We say that a set of positive integers $S$ is {\em Gowers uniform} if there exists a positive constant $c$ such that 
$$\lim_{N\to\infty}\norm{{\bf 1}_S-c}_{U_r(\Z_N)}=0$$ for every $r\in \N.$
\end{definition*}

\begin{remark*}
Note that if such a constant $c$ exists, then, for $r=1,$ it is equal to the density of $S,$ i.e., $$c=\lim_{N\to \infty}\frac{|S\cap [1,N]|}{N}.$$
\end{remark*}


Using a similar argument to the one of Theorem~\ref{T:t6}, we get the following:

\begin{proposition}\label{P:6}
Let $\ell, m\in \N,$ $(X,\mathcal{X},\mu,T_1,\ldots, T_\ell)$ be a system, $q_{i,j}\in\R[t]$ polynomials, $1\leq i\leq \ell,$ $1\leq j\leq m$ and  $f_1,\ldots,f_m\in L^{\infty}(\mu).$

For the sequence of functions $$b(n)= (\prod_{i=1}^\ell T_i^{[q_{i,1}(n)]})f_1\cdot\ldots\cdot(\prod_{i=1}^\ell T_i^{[q_{i,m}(n)]})f_m,$$ there exists $d\in \N,$ depending only on the maximum degree of the polynomials $q_{i,j}$ and the integers $\ell$ and $m,$ such that for every sequence $(a(n))$ with $\norm{a}_{\infty}\leq 1$ and every $0<\delta<1,$ there exists a constant $C_{d,\delta}$ depending on $d$ and $\delta,$ such that  $$\norm{\frac{1}{N}\sum_{n=1}^N a(n) \cdot b(n)}_{L^2(\mu)}\leq C_{d,\delta}\left(\norm{a\cdot {\bf 1}_{[1,N]}}_{U_d(\Z_{dN})}+o_N(1)\right)+c_{\delta},$$
 where $c_\delta\to 0$ as $\delta\to 0^+$ and the term $o_N(1)$ depends only on the integer $d.$
\end{proposition}

\begin{proof}
Let $0<\delta<1.$
For the given transformations on $X,$ we define the same action on $\R^{\ell m}$ as in Theorem~\ref{T:t6}, the function $\hat{f}_0$ and the same $Y$-extensions $\hat{f}_i,$ for $1\leq j\leq m,$ in the product probability space $Y=X\times [0,1)^{\ell m},$ endowed with the measure $\nu=\mu\times \lambda^{\ell m}.$

 We also define the same sequences of functions $(\tilde{b}(n))$ and $(b'(n))$ in $Y$ and $X$ respectively, as in the proof of Theorem~\ref{T:t6}. Then, by using the triangle and the Cauchy-Schwarz inequality, following the proof of Theorem~\ref{T:t6}, we have that 

$$\delta^{\ell m}\norm{\frac{1}{N}\sum_{n=1}^N a(n)b(n)}_{L^2(\mu)}\leq \norm{\frac{1}{N}\sum_{n=1}^N a(n)\cdot (\delta^{\ell m}b(n)-b'(n))}_{L^2(\mu)}+\norm{\frac{1}{N}\sum_{n=1}^N a(n)\tilde{b}(n)}_{L^2(\nu)}.$$ For the second term we can use Lemma~\ref{L:n1} to find $d=d(\ell, m, \max \deg (q_{i,j})),$ such that  $$\norm{\frac{1}{N}\sum_{n=1}^N a(n)\tilde{b}(n)}_{L^2(\nu)}\leq C_d\left(\norm{a\cdot {\bf 1}_{[1,N]}}_{U_d(\Z_{dN})}+o_N(1)\right),$$ Where $C_d$ is a constant depending on $d.$ We remark at this point that since the weighted sequence $(a(n))$ is bounded, then the $o_N(1)$ term of Lemma~\ref{L:n1} only depends on $d.$

Following the proof of Theorem~\ref{T:t6} we obtain (since $(a(n))$ is bounded by $1$)
$$\norm{\frac{1}{N}\sum_{n=1}^N a(n)\cdot (\delta^{\ell m}b(n)-b'(n))}_{L^2(\mu)}\leq \delta^{\ell m} c_\delta,$$
with $c_\delta\to 0$ as $\delta\to 0^+,$ since, independently of $x\in X,$ we have that $$\Big| \frac{1}{N}\sum_{n=1}^N(\delta^{\ell m}b(n)(x)-b'(n)(x))\Big|\leq \delta^{\ell m}c_\delta,$$ with $c_\delta\to 0$ as $\delta\to 0^+.$ The result now follows.
\end{proof}



We are now ready to prove Theorem~\ref{P:2} and close this article.

\begin{proof}[Proof of Theorem~\ref{P:2}] Let $\ell, m\in \N,$ 
$(X, \mathcal{X} , \mu, T_1,\ldots,$ $ T_\ell)$ system,
functions $f_1,\ldots, f_m \in L^\infty(\mu),$ polynomials $q_{i,j}\in \R[t]$ for $1\leq i\leq \ell,$ $1\leq j\leq m,$
and for $n\in \N$ let
\begin{equation}\label{E:fin}
b(n)= (\prod_{i=1}^\ell T_i^{[q_{i,1}(n)]})f_1\cdot\ldots\cdot (\prod_{i=1}^\ell T_i^{[q_{i,m}(n)]})f_m.
\end{equation} If $S\subseteq\N$ is a Gowers uniform set and $c$ is the respective constant from the definition, by Proposition~\ref{P:6}, we have that there exists $d\in \N,$ depending only on the maximum degree of the polynomials $q_{i,j}$ and the integers $\ell$ and $m,$ such that for every $0<\delta<1$ there exists a constant $C_{d,\delta}$ depending on $d$ and $\delta,$ such that
\begin{equation}\label{E:lou}
\norm{\frac{1}{N}\sum_{n=1}^N ({\bf 1}_S(n)-c) \cdot b(n)}_{L^2(\mu)}\leq C_{d,\delta}\left(\norm{({\bf 1}_S-c)\cdot {\bf 1}_{[1,N]}}_{U_d(\Z_{dN})}+o_N(1)\right)+c_{\delta},
\end{equation}
 where  $c_\delta\to 0$ as $\delta\to 0^+$ and the term $o_N(1)$ depends only on the integer $d.$


From Gowers uniformity, for $r=d,$ we have that $\lim_{N\to\infty}\norm{{\bf 1}_S-c}_{U_d(\Z_{dN})}=0.$ Then, using Lemma~\ref{L:nh}, for $\kappa=d,$ $\tilde{N}=dN$ and $J=[1,N],$ we get 
 $$\lim_{N\to\infty}\norm{({\bf 1}_S-c)\cdot {\bf 1}_{[1,N]}}_{U_d(\Z_{dN})}=0.$$

So, from \eqref{E:lou}, we have that
$$
\limsup_{N\to\infty}\norm{\frac{1}{N}\sum_{n\in S\cap [1,N]}b(n) -\frac{c}{N}\sum_{n=1}^N b(n)}_{L^2(\mu)}\leq c_\delta,
$$ 
hence, by taking $\delta\to 0^+,$ we get 
$$
\lim_{N\to\infty}\norm{\frac{1}{N}\sum_{n\in S\cap [1,N]}b(n) -\frac{c}{N}\sum_{n=1}^N b(n)}_{L^2(\mu)}=0.
$$
Using the remark following the definition of a Gowers uniform set, $c$ is the density of $S,$ so  $$\lim_{N\to\infty}\frac{|S\cap[1,N]|}{N}=c,$$
hence,
 \begin{equation}\label{E:fi}
 \lim_{N\to\infty}\norm{\frac{1}{|S\cap [1,N]|}\sum_{n\in S\cap [1,N]}b(n)-\frac{1}{N}\sum_{n=1}^N b(n)}_{L^2(\mu)}=0.
\end{equation}

By using \eqref{E:fi} and Theorem~\ref{T:ref}, we get the integer part polynomial multiple mean convergence conclusion. 

Similarly, letting $(b(n))$ being the respective expression to \eqref{E:fin} with the closest integer polynomial iterates and $f_i={\bf 1}_A,$ $1\leq i\leq m,$ where $A\in \mathcal{X}$ with $\mu(A)>0,$ by using \eqref{E:fi} and Theorem~\ref{T:ref}, we get the closest integer polynomial multiple recurrence result.
\end{proof}


\begin{thebibliography}{9999}

\bibitem{Be87a} V.~Bergelson. Weakly mixing PET. {\em Ergodic Theory
Dynam. Systems} \textbf{7} (1987), no. 3, 337--349.

\bibitem{BHa} V.~Bergelson, I.~J.~H{\aa}land. Sets of Recurrence and Generalized polynomials. {\em Convergence
in ergodic theory and probability} (Columbus, OH, 1993), Ohio State Univ. Math.
Res. Inst. Publ., de Gruyter, Berlin. 5 (1996), 91--110.

\bibitem{BHMP} V.~Bergelson, B.~Host, R.~McCutcheon, F.Parreau. Aspects of uniformity in recurrence. {\em Collog. Math.} \textbf{84/85} (2000), no. 2, 549--576.

\bibitem{BL2} V.~Bergelson, A.~Leibman. Polynomial extensions of van der Waerden's and Szemer{\'e}di's theorems. {\em J. Amer. Math. Soc.} \textbf{9} (1996), 725--753.

\bibitem{BLZ} V.~Bergelson, A.~Leibman, T.~Ziegler. The shifted primes and the multidimensional Szemer{\'e}di and polynomial van der Waerden Theorems. {\em C. R. Math. Acad. Sci. Paris.} \textbf{349} (2011), Issues 3-4, 123--125.

\bibitem{BMc} V.~Bergelson, R.~McCucheon. Idempotent Ultrafilters, Multiple Weak Mixing
and Szemer\'edi's Theorem for Generalized Polynomials. {\em J. Analyse Math.} \textbf{111} (2010), 77--130.

\bibitem{BJW} M.~Boshernitzan, R.~L.~Jones, M.~Wierdl. Integer and fractional parts of good averaging sequences in ergodic theory. {\em Convergence in Ergodic Theory and Probability (Caribbean Studies),} Jul 1, 1996.

\bibitem{Fr11}
N.~Frantzikinakis.
Some open problems on multiple ergodic averages.
  \texttt{arXiv:1103.3808}.
  
\bibitem{F}
N.~Frantzikinakis. Multiple correlation sequences and nilsequences. {\em Invent. Math.} {\bf 202} (2015), no. 2, 875--892.

\bibitem{FranH2}
N. Frantzikinakis, B. Host. Asymptotics for multilinear averages of multiplicative functions. To appear in Mathematical Proceedings of the Cambridge Philosophical Society.

\bibitem{FranH}
N.~Frantzikinakis, B.~Host. Multiple ergodic theorems for arithmetic sets. To appear in  Trans. Amer. Math. Soc.

\bibitem{FHK2} N.~Frantzikinakis, B.~Host, B.~Kra. Multiple recurrence and convergence for sequences related to the prime numbers. {\em J. Reine Angew. Math.} \textbf{611} (2007), 131--144.

\bibitem{FHK} N.~Frantzikinakis, B.~Host, B.~Kra. The polynomial multidimensional Szemer\'edi theorem along shifted primes. {\em Israel J. Math.} \textbf{194} (2013),  no. 1, 331--348.

\bibitem{Fu} H.~Furstenberg. Ergodic behaviour of diagonal measures and a theorem of Szemer{\'e}di on arithmetic progressions. {\em J. Analyse Math.} \textbf{71} (1977), 204--256.

\bibitem{G} W.~Gowers. A new proof of Szemer{\'e}di's theorem. {\em Geom. Funct. Anal.} \textbf{11} (2001), 465--588.

\bibitem{GT} B.~Green, T.~Tao. Linear equations in primes. {\em Annals. Math.} \textbf{171} (2010), 1753--1850.

\bibitem{GT2} B.~Green, T.~Tao. The M{\"o}bius function is strongly orthogonal to nilsequences. To appear, {\em Annals. Math.} 

\bibitem{GTZ} B.~Green, T.~Tao, T.~Ziegler. An inverse theorem for the Gowers $U^{s+1}[N]$-norm. {\em Annals. Math.} {\bf 176} (2012), no. 2, 1231--1372.

\bibitem{K} A.~Koutsogiannis. Integer part correlation sequences. \texttt{arXiv:1512.01313}.

\bibitem{EL} E.~Lesigne. On the sequence of integer parts of a good sequence for the ergodic theorem. {\em Comment. Math. Univ. Carolin.} \textbf{36} (1995), no. 4, 737--743.

\bibitem{Sark} A.~S{\'a}rk{\"o}zy. On difference sets of sequences of integers, III. {\em Acta Math. Acadm. Sci. Hungar.} \textbf{31} (1978), 355--386.

\bibitem{WS} W.~Sun. Multiple recurrence and convergence for certain averages along shifted primes. {\em Ergod. Th. Dynam. Sys.} \textbf{35} (2014), 1592--1609.

\bibitem{W12} M.~Walsh. Norm convergence of nilpotent ergodic averages. {\em Annals of Mathematics} \textbf{175} (2012), no. 3, 1667--1688.

\bibitem{Mate2} M.~Wierdl. Pointwise ergodic theorem along the prime numbers. {\em Israel J. Math.} \textbf{64} (1988), 315--336.

\bibitem{Mate} M.~Wierdl. Personal communication.

\bibitem{WZ} T.~Wooley, T.~Ziegler. Multiple recurrence and convergence along the primes. {\em Amer. J. of Math.} \textbf{134} (2012), no. 6, 1705--1732.
\end{thebibliography}
\end{document}